\documentclass[11pt,a4paper]{article}

\usepackage[T1]{fontenc}
\usepackage{lmodern}
\usepackage{microtype}
\usepackage{amsmath,amssymb,amsthm,mathtools}
\usepackage{enumitem}
\usepackage{authblk}
\usepackage[margin=1in]{geometry}
\usepackage{xcolor}
\usepackage[colorlinks=true,linkcolor=blue!55!black,citecolor=blue!55!black,urlcolor=blue!55!black]{hyperref}

\newtheorem{thm}{Theorem}[section]

\newtheorem{lemma}[thm]{Lemma}

\newtheorem{conj}[thm]{Conjecture}

\newtheorem{example}{Example}[section]
\newtheorem{defin}{Definition}[section]

\begin{document}
\renewcommand{\baselinestretch}{1.3}
%%%%%%%%%%%%%%%%%%%%%%%%%%%%%%%%%%%%%%%%%%%%%%%%%%%%%%%%%%%%%%%%%%%%%%%%%%%%%%%%%%%%%%%%
%%%%%%%%%%%%%%%%%%%%%%%%%%%%%%%%%%%%%%%%%%%%%%%%%%%%%%%%%%%%%%%%%%%%%%%%%%%%%%%%%%%%%%%%

\title{\bf Counting sunflowers with restricted matching number}

\author[1]{Haixiang Zhang\thanks{E-mail: \texttt{zhang-hx22@mails.tsinghua.edu.cn}}}
\author[2]{Mengyu Cao\thanks{Corresponding author. E-mail: \texttt{myucao@ruc.edu.cn}}}
\author[1]{Mei Lu\thanks{E-mail: \texttt{lumei@tsinghua.edu.cn}}}

\affil[1]{\small Department of Mathematical Sciences, Tsinghua University, Beijing 100084, China}
\affil[2]{\small Institute for Mathematical Sciences, Renmin University of China, Beijing 100086, China}

\date{}
\maketitle
\begin{abstract}
For $\mathcal H\subseteq\binom{[n]}k$, let $\nu(\mathcal H)$ denote its matching number and let $d_{\mathcal H}(E)$ denote the codegree of a $(k-1)$-set $E$.  We study the codegree moments
$
 co_p(\mathcal H)=\sum_{E\in\binom{[n]}{k-1}}d_{\mathcal H}(E)^p
$
and the number of copies of the $k$-uniform $l$-petal sunflower $S_{k,l}^{k-1}$ whose core has size $k-1$.  For fixed $k,s,l$ and $p$, and all sufficiently large $n$, we prove that among all $k$-uniform families with matching number at most $s$, both quantities are uniquely maximized by
\[
 \mathcal H_{n,k,s}=\left\{F\in\binom{[n]}k:F\cap[s]\neq\emptyset\right\}.
\]
Thus the sunflower result is the generalized Tur\'an problem $\operatorname{ex}_k(n,S_{k,l}^{k-1},M_{s+1})$.  When $k=3$, we make the threshold effective: for every integer $p\geq1$, the codegree-moment conclusion holds for $n\geq6s+3$.  The proof combines a Hilton--Milner-type stability theorem of Frankl and Kupavskii with a high-codegree-shadow reduction and a convex moment estimate.


\end{abstract}

\section{Introduction}

Let $[n]:=\{1,2,\ldots,n\}$ and $\binom{[n]}k:=\{S\subseteq[n]:|S|=k\}$.  A family $\mathcal H\subseteq\binom{[n]}k$ is \emph{intersecting} if $F_1\cap F_2\neq\emptyset$ for all $F_1,F_2\in\mathcal H$.  A \emph{full star} is a family isomorphic to $\{F\in\binom{[n]}k:1\in F\}$.  The classical Erd\H{o}s--Ko--Rado theorem determines the maximum size of an intersecting uniform family.

\begin{thm}[Erd\H{o}s-Ko-Rado theorem, \cite{Erdos-Ko-Rado-1961-313}]\label{EKR}	
	Let $ n \geq 2k $ and $ \mathcal{H} \subseteq \binom{[n]}{k} $ be a $k$-uniform intersecting family. Then $ |\mathcal{H}| \leq \binom{n - 1}{k - 1} $. Moreover, if $ n > 2k $, then $ |\mathcal{H}| = \binom{n - 1}{k - 1} $ if and only if $ \mathcal{H}$ is a full star.	
\end{thm}

An intersecting family $\mathcal F\subseteq\binom{[n]}k$ is \emph{trivial} if $\bigcap_{F\in\mathcal F}F\neq\emptyset$.  The Hilton--Milner theorem determines the maximum size of a nontrivial intersecting family.  The Erd\H{o}s--Ko--Rado and Hilton--Milner theorems and their extensions have long been central topics in extremal set theory; see, for example, \cite{Ahlswede-Khachatrian-1997,Ahlswede1996,Frankl-1978-1,Frankl-Furedi-1986,Frankl--Furedi-1991,Wilson-1984}.

\begin{thm}[Hilton-Milner theorem, \cite{Hilton-Milner-1967}]\label{HM}
   Let $ n \geq 2k \geq 6 $ and $ \mathcal{H} \subseteq \binom{[n]}{k} $ be a $k$-uniform intersecting family. If $ \mathcal{H} $ is nontrivial, then $ |\mathcal{H}| \leq \binom{n - 1}{k - 1} - \binom{n - k - 1}{k - 1} + 1 $. Moreover, if $ n > 2k > 6 $, then $ |\mathcal{H}| = \binom{n - 1}{k - 1} - \binom{n - k - 1}{k - 1} + 1 $ if and only if $ \mathcal{H} \cong \{F\in\binom{[n]}{k}:1\in F,F\cap[2,k+1]\neq\emptyset\}\cup\{[2,k+1]\}$.
\end{thm}

A \emph{matching} in $\mathcal H\subseteq\binom{[n]}k$ is a collection of pairwise disjoint members, and the \emph{matching number} $\nu(\mathcal H)$ is the maximum size of such a collection.  Thus intersecting families are precisely the nonempty families with matching number one.  The Erd\H{o}s Matching Conjecture asks for the maximum size of a uniform family with bounded matching number.

For $1\leq i \leq k$, let $$\mathcal{A}_{n,k,s,i}=\left\{F\in \binom{[n]}{k}:|F\cap [is+i-1]|\geq i\right\},$$ and  $$\mathcal{H}_{n,k,s}=\mathcal{A}_{n,k,s,1}=\left\{F\in\binom{[n]}{k}:F\cap[s]\neq\emptyset\right\}.$$ Note that $\nu(\mathcal{A}_{n,k,s,i})=\nu(\mathcal{H}_{n,k,s})=s$ for $1\leq i \leq k$.
\begin{conj}[Erd\H{o}s matching conjecture]\label{EMC}
	Let $\mathcal{F}\subseteq \binom{[n]}{k}$ satisfy $\nu(\mathcal{F})\leq s$ and $n\geq ks+k-1$. Then
    \[
    |\mathcal{F}|\leq \max\left\{\binom{n}{k}-\binom{n-s}{k},\binom{ks+k-1}{k}\right\}.
    \]
\end{conj}

The conjecture is known for sufficiently large $n$ \cite{Erdos1959,frankl2013,frankl2017maximum,luczak2014erdHos} and has been completely resolved for $k=3$.  {\L}uczak and Mieczkowska \cite{luczak2014erdHos} proved the latter assertion for sufficiently large $s$, and Frankl \cite{frankl2017maximum} completed the proof for every $s$.
\begin{thm}[Frankl \cite{frankl2017maximum}]\label{EMC k=3}
    Conjecture \ref{EMC} is true for $k=3$.
\end{thm}
For general $k$, the conjecture remains open.  In the large-$n$ range, the first candidate $\mathcal H_{n,k,s}=\mathcal A_{n,k,s,1}$ is larger than the clique candidate.  We use the following theorem in that range.

\begin{thm}[Frankl \cite{frankl2013}]\label{result for EMC}
	Let $\mathcal F\subseteq\binom{[n]}k$ satisfy $\nu(\mathcal F)\leq s$.
	If
	$
	n\geq(2s+1)k-s,
	$
	then
	\[
	|\mathcal F|
	\leq|\mathcal H_{n,k,s}|
	=\binom nk-\binom{n-s}k.
	\]
\end{thm}

For a family $\mathcal F$, let $\tau(\mathcal F)$ be the minimum size of a set meeting every member of $\mathcal F$.  We shall also use the following Hilton--Milner-type matching theorem.

\begin{thm}[Frankl-Kupavskii \cite{frankl-kupavskii2019}]\label{F-K}
Let $k\geq3$ and
\[
n\geq \bigl(s+\max\{25,2s+2\}\bigr)k.
\]
If $\mathcal F\subseteq\binom{[n]}k$ satisfies $\nu(\mathcal F)=s$ and $\tau(\mathcal F)>s$, then
\[
|\mathcal F|\leq \binom nk-\binom{n-s}k+1-\binom{n-s-k}{k-1}.
\]
\end{thm}

For $k$-uniform hypergraphs $\mathcal G$ and $\mathcal F$, let $N(\mathcal G,\mathcal H)$ be the number of unlabelled copies of $\mathcal G$ in $\mathcal H$, and define
\[
\operatorname{ex}_k(n,\mathcal G,\mathcal F)
:=\max\left\{N(\mathcal G,\mathcal H):\mathcal H\subseteq\binom{[n]}k
\text{ is $\mathcal F$-free}\right\}.
\]
Since an $M_{s+1}$-free family has matching number at most $s$, the Erd\H{o}s Matching Conjecture concerns $\operatorname{ex}_k(n,M_{s+1})$.  We study the analogous counting problem.  A $k$-uniform \emph{sunflower} with $l$ petals is a collection of distinct sets $\{A_1,\ldots,A_l\}$ for which there is a set $C$, called the \emph{core}, such that $A_i\cap A_j=C$ whenever $i\neq j$.  We write $S_{k,l}^{k-1}$ when $|C|=k-1$.  In particular,
\[
N(S_{k,l}^{k-1},\mathcal H)
=\sum_{E\in\binom{[n]}{k-1}}\binom{d_{\mathcal H}(E)}l.
\]

\begin{thm}\label{main1}
    Let $k,s,l$ be positive integers with $s\geq 1$ and $k,l\geq 2$.  There exists $n_0=n_0(k,s,l)$ such that, whenever $n\geq n_0$ and $\mathcal{H}\subseteq \binom{[n]}{k}$ satisfies $\nu(\mathcal{H})\leq s$,
    $$N(S_{k,l}^{k-1},\mathcal{H})\leq N(S_{k,l}^{k-1},\mathcal{H}_{n,k,s})$$
    with equality if and only if $\mathcal{H}\cong \mathcal{H}_{n,k,s}$.
\end{thm}

For $E\in\binom{[n]}t$, where $1\leq t<k$, define
\[
d_{\mathcal H}(E):=|\{F\in\mathcal H:E\subseteq F\}|.
\]
For a positive integer $p$, put
\[
co_p(\mathcal H):=\sum_{E\in\binom{[n]}{k-1}}d_{\mathcal H}(E)^p.
\]
Thus $co_1(\mathcal H)=k|\mathcal H|$ and, with the usual unlabelled-copy convention,
\[
co_2(\mathcal H)=2N(S_{k,2}^{k-1},\mathcal H)+k|\mathcal H|.
\]
Balogh, Clemen and Lidick\'y \cite{balogh2022} initiated the systematic study of hypergraph Tur\'an problems for degree moments.  Brooks and Linz \cite{brooks2023some} determined the maximum of $co_2$ under a matching-number constraint; our next result treats every positive integral moment.

\begin{thm}[Brooks--Linz \cite{brooks2023some}]
    Let $n,k,s$ be positive integers with $s\geq1$ and $k\geq2$, and let $\mathcal H\subseteq\binom{[n]}k$ satisfy $\nu(\mathcal H)\leq s$.  There exists an integer $n_0$ such that, for $n\geq n_0$,
    \[
    co_2(\mathcal H)\leq co_2(\mathcal H_{n,k,s}),
    \]
    with equality if and only if $\mathcal H\cong\mathcal H_{n,k,s}$.
\end{thm}

\begin{thm}\label{main2}
    Let $k,s,p$ be positive integers with $s\geq1$.  There exists $n_0=n_0(k,s,p)$ such that, whenever $n\geq n_0$ and $\mathcal H\subseteq\binom{[n]}k$ satisfies $\nu(\mathcal H)\leq s$,
    \[
    co_p(\mathcal H)\leq co_p(\mathcal H_{n,k,s}),
    \]
    with equality if and only if $\mathcal H\cong\mathcal H_{n,k,s}$.
\end{thm}

When $p=1$, the upper bound in Theorem~\ref{main2} follows from
Theorem~\ref{result for EMC}; the equality case will follow from the
stability argument in Section~\ref{section 2}.

Section~\ref{section 2} proves stability forms of Theorems~\ref{main1} and~\ref{main2}.  The resulting threshold $n_0(k,s,p)$ is qualitative.  For $k=3$, the next theorem makes it explicit and uniform in $p$: in particular, $n_0(3,s,p)\leq6s+3$.
\begin{thm}\label{th k=3}
    Let $n,s,p$ be positive integers, and let $\mathcal H\subseteq\binom{[n]}3$ satisfy $\nu(\mathcal H)\leq s$.  If $n\geq6s+3$, then
    \[
    co_p(\mathcal H)\leq co_p(\mathcal H_{n,3,s}),
    \]
    with equality if and only if $\mathcal H\cong\mathcal H_{n,3,s}$.
\end{thm}

Wang and Peng \cite{W-P} determined the maximum number of copies of
$S_{k,2}^{k-1}$ under a matching-number constraint when
$n\geq 2sk^3$.  Since every $k$-uniform family with matching number at
most $s$ has $O_{k,s}(n^{k-1})$ edges, the identity
\[
co_2(\mathcal H)
=
2N(S_{k,2}^{k-1},\mathcal H)+k|\mathcal H|
\]
shows that their sunflower-counting result also determines the leading
term of the corresponding second codegree moment.  Our results extend this leading-order correspondence
to exact extremal theorems in two directions: Theorem~\ref{main1}
determines the maximum number of copies of $S_{k,l}^{k-1}$, together
with the unique extremal family, for every $l\geq2$, while
Theorem~\ref{main2} treats every positive integral codegree moment.
Moreover, when $k=3$, Theorem~\ref{th k=3} provides the explicit
linear threshold $n\geq6s+3$, uniformly for all $p$.
\section{Proofs of the main theorems}\label{section 2}
We first prove two stability statements.  They show that a family with nearly extremal sunflower count or codegree moment is covered by at most $s$ points.  This structural conclusion reduces the extremal problem to the canonical family $\mathcal H_{n,k,s}$.

\begin{thm}\label{sta main1}
    Let $k,s,l$ be positive integers with $s\geq1$ and $k,l\geq2$.  For every $\epsilon>0$ there exists $n_\epsilon$ such that, whenever $n\geq n_\epsilon$ and $\mathcal H\subseteq\binom{[n]}k$ satisfies $\nu(\mathcal H)\leq s$ and
    \[
    N(S_{k,l}^{k-1},\mathcal H)\geq
    \left(\frac{s-1}{(k-2)!l!}+\epsilon\right)n^{k+l-2},
    \]
    the family $\mathcal H$ is the union of at most $s$ trivial intersecting families.
\end{thm}
\begin{thm}\label{sta main2}
    Let $k,s,p$ be positive integers with $s\geq1$ and $k,p\geq2$.  For every $\epsilon>0$ there exists $n_\epsilon$ such that, whenever $n\geq n_\epsilon$ and $\mathcal H\subseteq\binom{[n]}k$ satisfies $\nu(\mathcal H)\leq s$ and
    \[
    co_p(\mathcal H)\geq
    \left(\frac{s-1}{(k-2)!}+\epsilon\right)n^{k+p-2},
    \]
    the family $\mathcal H$ is the union of at most $s$ trivial intersecting families.
\end{thm}

For $\mathcal H\subseteq\binom{[n]}k$ and $d\geq1$, write
\[
\mathcal K_d(\mathcal H):=\left\{K\in\binom{[n]}{k-1}:d_{\mathcal H}(K)\geq d\right\}.
\]
The proof first obtains a covering stability statement for large families.  A large codegree moment then forces $\mathcal K_{sk+1}(\mathcal H)$ to be large, so the stability statement controls this high-codegree shadow.  Finally, the cover of the shadow is lifted to a cover of $\mathcal H$.

\begin{lemma}[Stability of Erd\H{o}s matching conjecture]\label{lemma sta}
    Let $k,s$ be positive integers.  For every $\epsilon>0$ there exists $n'_\epsilon$ such that, whenever $n\geq n'_\epsilon$ and $\mathcal H\subseteq\binom{[n]}k$ satisfies $\nu(\mathcal H)\leq s$ and
    \[
    |\mathcal H|\geq\left(\frac{s-1}{(k-1)!}+\epsilon\right)n^{k-1},
    \]
    the family $\mathcal H$ is the union of at most $s$ trivial intersecting families.
\end{lemma}
\begin{proof}
	It suffices to prove that $\tau(\mathcal H)\leq s$.  When $k=1$, this
	is immediate, since $\mathcal H$ consists of at most $s$ singletons.
	
	We first establish the following estimate for graphs.  If a graph $G$
	satisfies $\nu(G)\leq s$ and $\tau(G)>s$, then
	\begin{equation}\label{graph stability estimate}
		e(G)\leq
		(s-1)|V(G)|+\binom{4s^2+2s}{2}.
	\end{equation}
	We prove this by induction on $s$.  If $s=1$, then $G$ is an
	intersecting graph with no one-vertex cover.  Such a graph is contained
	in a triangle, and hence
	$
	e(G)\leq3\leq\binom{6}{2}.
	$
	
	Now let $s\geq2$ and suppose that the assertion holds with $s-1$ in
	place of $s$.  If there is a vertex $v$ such that
	$\nu(G-v)\leq s-1$, then $\tau(G-v)>s-1$, since otherwise a cover of
	$G-v$ of size at most $s-1$, together with $v$, would cover $G$.
	Therefore, by induction,
	\[
	\begin{aligned}
		e(G)
		&=e(G-v)+d_G(v)\\
		&\leq
		(s-2)(|V(G)|-1)
		+\binom{4(s-1)^2+2(s-1)}{2}
		+|V(G)|-1\\
		&=
		(s-1)(|V(G)|-1)
		+\binom{4(s-1)^2+2(s-1)}{2}\\
		&\leq
		(s-1)|V(G)|+\binom{4s^2+2s}{2}.
	\end{aligned}
	\]
	
	It remains to consider the case in which $\nu(G-v)=s$ for every
	$v\in V(G)$.  Fix a matching $M$ of size $s$ and put $U=V(M)$.  For
	each $u\in U$, choose a matching $M_u$ of size $s$ in $G-u$, and set
	\[
	W:=U\cup\bigcup_{u\in U}V(M_u).
	\]
	Since $|U|=2s$, we have
	$
	|W|\leq2s+2s\cdot2s=4s^2+2s.
	$
	We claim that every vertex outside $W$ is isolated.  Indeed, suppose
	that $xy\in E(G)$ with $x\notin W$.  If $y\notin U$, then
	$M\cup\{xy\}$ is a matching of size $s+1$, contradicting
	$\nu(G)\leq s$.  Hence $y\in U$.  Since $M_y$ is a matching in $G-y$
	and $x\notin V(M_y)$, the set $M_y\cup\{xy\}$ is again a matching of
	size $s+1$, a contradiction.  Thus every edge of $G$ is contained in
	$W$, and consequently
	\[
	e(G)\leq\binom{|W|}{2}
	\leq\binom{4s^2+2s}{2}.
	\]
	This proves \eqref{graph stability estimate}.
	
	When $k=2$, regard $\mathcal H$ as a graph on $[n]$.  If
	$\tau(\mathcal H)>s$, then \eqref{graph stability estimate} gives
	\[
	|\mathcal H|
	\leq
	(s-1)n+\binom{4s^2+2s}{2},
	\]
	which contradicts
	$
	|\mathcal H|\geq(s-1+\epsilon)n
	$
	for all sufficiently large $n$.  Hence
	$\tau(\mathcal H)\leq s$ when $k=2$.
	
	We now assume that $k\geq3$.  We first show that
	$\nu(\mathcal H)=s$.  If $s=1$, then the assumed lower bound implies
	that $\mathcal H$ is nonempty, and hence $\nu(\mathcal H)=1$.
	Suppose that $s\geq2$.  If $\nu(\mathcal H)\leq s-1$, then
	Theorem~\ref{result for EMC}, applied with matching bound $s-1$, gives
	\[
	\begin{aligned}
		|\mathcal H|
		\leq|\mathcal H_{n,k,s-1}|
		=\binom{n}{k}-\binom{n-s+1}{k}
		=
		\left(
		\frac{s-1}{(k-1)!}+O_{k,s}(n^{-1})
		\right)n^{k-1}
	\end{aligned}
	\]
	for all sufficiently large $n$.  This contradicts the assumed lower
	bound.  Therefore $\nu(\mathcal H)=s$.
	
	Suppose, for a contradiction, that $\tau(\mathcal H)>s$.  For all
	sufficiently large $n$, Theorem~\ref{F-K} gives
	\[
	\begin{aligned}
		|\mathcal H|
		\leq
		\binom{n}{k}-\binom{n-s}{k}
		+1-\binom{n-s-k}{k-1}
		=
		\left(
		\frac{s-1}{(k-1)!}+O_{k,s}(n^{-1})
		\right)n^{k-1},
	\end{aligned}
	\]
	again contradicting the assumed lower bound.  Hence
	$\tau(\mathcal H)\leq s$.
	
	Finally, let $X$ be a cover of $\mathcal H$ with $|X|\leq s$.  Then
	$
	\mathcal H
	=
	\bigcup_{x\in X}
	\{F\in\mathcal H:x\in F\},
	$
	and each family in this union is trivially intersecting.  Therefore
	$\mathcal H$ is the union of at most $s$ trivial intersecting
	families.
\end{proof}

The next lifting lemma relates the matching number of a high-codegree shadow to that of the original family.

\begin{lemma} \label{lem1}
    Let $k\geq2$, $s\geq1$, and $\mathcal H\subseteq\binom{[n]}k$.  If $\nu(\mathcal H)\leq s$, then $\nu(\mathcal K_{sk+1}(\mathcal H))\leq s$.
\end{lemma}
\begin{proof}
Suppose that $F_1,\ldots,F_{s+1}$ are pairwise disjoint members of $\mathcal K_{sk+1}(\mathcal H)$.  We greedily choose pairwise disjoint $G_1,\ldots,G_{s+1}\in\mathcal H$ with $F_i\subset G_i$.  Assume that $G_1,\ldots,G_m$ have been chosen and are disjoint from $F_{m+1},\ldots,F_{s+1}$.  An extension of $F_{m+1}$ is forbidden only when its additional vertex lies in
\[
\bigcup_{i=1}^mG_i\ \cup\!\bigcup_{i=m+2}^{s+1}F_i,
\]
which contains at most $mk+(s-m)(k-1)$ vertices.  Consequently, the number of admissible extensions is at least
\[
sk+1-mk-(s-m)(k-1)=s-m+1>0.
\]
This completes the greedy construction and produces an $(s+1)$-matching in $\mathcal H$, a contradiction.
\end{proof}

A large codegree moment forces the high-codegree shadow to be large.

\begin{lemma}\label{lem2}
     Let $k,p\geq2$, $s\geq1$, and $\epsilon>0$.  For all sufficiently large $n$, if $\mathcal H\subseteq\binom{[n]}k$ satisfies
     \[
     co_p(\mathcal H)\geq
     \left(\frac{s-1}{(k-2)!}+\epsilon\right)n^{k+p-2},
     \]
     then
     \[
     |\mathcal K_{sk+1}(\mathcal H)|\geq
     \left(\frac{s-1}{(k-2)!}+\frac{\epsilon}{2}\right)n^{k-2}.
     \]
\end{lemma}
\begin{proof}
Put $\mathcal K=\mathcal K_{sk+1}(\mathcal H)$.  Since every $(k-1)$-set outside $\mathcal K$ has codegree at most $sk$,
\begin{align*}
co_p(\mathcal H)
&\leq |\mathcal K|(n-k+1)^p+(sk)^p\binom n{k-1}\\
&\leq |\mathcal K|n^p+O(n^{k-1}).
\end{align*}
Because $p\geq2$, the error term is $o(n^{k+p-2})$.  The asserted lower bound for $|\mathcal K|$ follows for sufficiently large $n$.
\end{proof}

\begin{lemma}\label{lem3}
     Let $k\geq2$, $s\geq1$, and $\epsilon>0$.  For all sufficiently large $n$, suppose that $\mathcal H\subseteq\binom{[n]}k$ satisfies $\nu(\mathcal H)\leq s$ and
     \[
     |\mathcal K_{sk+1}(\mathcal H)|\geq
     \left(\frac{s-1}{(k-2)!}+\epsilon\right)n^{k-2}.
     \]
     Then $\mathcal H$ is the union of at most $s$ trivial intersecting families.
\end{lemma}
\begin{proof}
Set $\mathcal K=\mathcal K_{sk+1}(\mathcal H)$.  Lemma~\ref{lem1} gives $\nu(\mathcal K)\leq s$, and Lemma~\ref{lemma sta}, applied with $\epsilon/2$, shows that $\tau(\mathcal K)\leq s$.  After a permutation of $[n]$, we may therefore assume
\[
\mathcal K\subseteq\mathcal H_{n,k-1,s}.
\]
It remains to prove that $\mathcal H\subseteq\mathcal H_{n,k,s}$.

Suppose instead that $G_0\in\mathcal H$ is disjoint from $[s]$, and put
\[
\mathcal H':=\{F\in\mathcal H:F\cap G_0=\emptyset\}.
\]
If $K\in\mathcal K$ is disjoint from $G_0$, at most $k$ extensions of $K$ in $\mathcal H$ meet $G_0$.  Hence
\[
d_{\mathcal H'}(K)\geq sk+1-k=(s-1)k+1,
\]
so $K\in\mathcal K_{(s-1)k+1}(\mathcal H')$.

Every member of $\mathcal K$ meets $[s]$.  Consequently, when $k\geq3$, the number of its members that meet $G_0$ is at most
\[
sk\binom{n-2}{k-3}=O(n^{k-3});
\]
when $k=2$, this number is zero because $G_0\cap[s]=\emptyset$.  It follows that
\begin{align*}
|\mathcal K_{(s-1)k+1}(\mathcal H')|
&\geq\left(\frac{s-1}{(k-2)!}+\epsilon\right)n^{k-2}-O(n^{k-3})\\
&>|\mathcal H_{n,k-1,s-1}|
\end{align*}
for all sufficiently large $n$.  If $s\geq2$, Theorem~\ref{result for EMC}, with uniformity $k-1$ and matching bound $s-1$, yields
\[
\nu\bigl(\mathcal K_{(s-1)k+1}(\mathcal H')\bigr)\geq s.
\]
Lemma~\ref{lem1}, applied with parameter $s-1$, then gives $\nu(\mathcal H')\geq s$.  For $s=1$, the displayed shadow is nonempty, which directly implies the same conclusion.  Thus $\mathcal H'$ contains disjoint edges $G_1,\ldots,G_s$, and these edges together with $G_0$ form an $(s+1)$-matching in $\mathcal H$, a contradiction.
\end{proof}

\noindent\textit{Proof of Theorem \ref{sta main2}.}
Lemma~\ref{lem2} gives
\[
|\mathcal K_{sk+1}(\mathcal H)|\geq
\left(\frac{s-1}{(k-2)!}+\frac{\epsilon}{2}\right)n^{k-2}.
\]
The conclusion now follows from Lemma~\ref{lem3}, used with $\epsilon/2$. \hfill$\square$

\noindent\textit{Proof of Theorem \ref{sta main1}.}
For every nonnegative integer $d$,
$d^l\geq l!\binom dl$.  Therefore
\begin{align*}
co_l(\mathcal H)
&\geq l!\sum_{E\in\binom{[n]}{k-1}}\binom{d_{\mathcal H}(E)}l\\
&=l!N(S_{k,l}^{k-1},\mathcal H)\\
&\geq\left(\frac{s-1}{(k-2)!}+\epsilon l!\right)n^{k+l-2}.
\end{align*}
Theorem~\ref{sta main2} completes the proof. \hfill$\square$

\noindent\textit{Proof of Theorem \ref{main2}.}
The case $k=1$ is immediate, since then $co_p(\mathcal H)=|\mathcal H|^p$ and $|\mathcal H|=\nu(\mathcal H)\leq s$.  For $p=1$, the inequality follows from Theorem~\ref{result for EMC}.  If equality holds, then
\[
|\mathcal H|=|\mathcal H_{n,k,s}|=
\left(\frac{s}{(k-1)!}+O(n^{-1})\right)n^{k-1},
\]
so Lemma~\ref{lemma sta}, used for example with $\epsilon=1/(2(k-1)!)$, gives $\tau(\mathcal H)\leq s$ for sufficiently large $n$.  Hence $\mathcal H\subseteq\mathcal H_{n,k,s}$ after a permutation; equality of their sizes proves uniqueness.  Assume that $k,p\geq2$, and let $\mathcal H$ maximize $co_p$ subject to $\nu(\mathcal H)\leq s$.  Since $\mathcal H_{n,k,s}$ is admissible,
\[
co_p(\mathcal H)\geq co_p(\mathcal H_{n,k,s})
=\left(\frac{s}{(k-2)!}+o(1)\right)n^{k+p-2}.
\]
For sufficiently large $n$, Theorem~\ref{sta main2}, with for example $\epsilon=1/(2(k-2)!)$, shows that $\tau(\mathcal H)\leq s$.  After a permutation of $[n]$, therefore,
\[
\mathcal H\subseteq\mathcal H_{n,k,s}.
\]
It follows immediately that $co_p(\mathcal H)\leq co_p(\mathcal H_{n,k,s})$.  Moreover, adding any missing edge strictly increases $co_p$, so equality forces $\mathcal H=\mathcal H_{n,k,s}$. \hfill$\square$

\noindent\textit{Proof of Theorem \ref{main1}.}
Let $\mathcal H$ maximize the number of copies of $S_{k,l}^{k-1}$ subject to $\nu(\mathcal H)\leq s$.  The canonical family is admissible and
\[
N(S_{k,l}^{k-1},\mathcal H_{n,k,s})
=\left(\frac{s}{(k-2)!l!}+o(1)\right)n^{k+l-2}.
\]
Thus Theorem~\ref{sta main1}, with for example $\epsilon=1/(2(k-2)!l!)$, gives $\tau(\mathcal H)\leq s$ for sufficiently large $n$.  After a permutation, $\mathcal H\subseteq\mathcal H_{n,k,s}$, and hence
\[
N(S_{k,l}^{k-1},\mathcal H)\leq
N(S_{k,l}^{k-1},\mathcal H_{n,k,s}).
\]
To verify uniqueness, take an edge $F\in\mathcal H_{n,k,s}$ and choose $u\in F\cap[s]$ and $v\in F\setminus\{u\}$.  The $(k-1)$-set $F\setminus\{v\}$ contains $u$ and consequently has codegree $n-k+1$ in $\mathcal H_{n,k,s}$.  For $n-k+1\geq l$, the edge $F$ therefore belongs to a copy of $S_{k,l}^{k-1}$.  Hence every proper subfamily of $\mathcal H_{n,k,s}$ misses at least one sunflower copy, proving the equality statement. \hfill$\square$

\section{The 3-uniform case}\label{section 3}
We prove Theorem~\ref{th k=3} by induction on the matching number.  The following recurrence for the canonical family will be used in the inductive case.  Separating the constant term avoids any need to define $co_0$.

\begin{lemma}\label{eq1}
Let $n,s,p$ be positive integers with $n\geq s+2$.  Then
\[
co_p(\mathcal H_{n,3,s})
=(n-1)(n-2)^p+\binom{n-1}{2}
+\sum_{j=1}^p\binom pj co_j(\mathcal H_{n-1,3,s-1}).
\]
\end{lemma}
\begin{proof}
Every pair containing $1$ has codegree $n-2$ in $\mathcal H_{n,3,s}$.  For $E\in\binom{[2,n]}2$, after identifying $[2,n]$ with $[n-1]$,
\[
d_{\mathcal H_{n,3,s}}(E)=d_{\mathcal H_{n-1,3,s-1}}(E)+1.
\]
Expanding the $p$th power and summing the constant term over the $\binom{n-1}{2}$ such pairs proves the identity.
\end{proof}

\begin{lemma}\label{ineq1}
Let $n,s,p$ be positive integers with $n\geq s+2$.  Then
\[
co_p(\mathcal H_{n,3,s})\geq
\left(\binom n2-\binom{n-s}2\right)(n-2)^p
\geq s(n-s)(n-2)^p.
\]
\end{lemma}
\begin{proof}
Each pair meeting $[s]$ has codegree $n-2$ in $\mathcal H_{n,3,s}$, and there are $\binom n2-\binom{n-s}2$ such pairs.  The second inequality follows from
\[
\binom n2-\binom{n-s}2=\frac{s(2n-s-1)}2\geq s(n-s).
\]
\end{proof}

For $i,j\in[n]$, write
\[
\mathcal H_{\bar i}:=\{F\in\mathcal H:i\notin F\},\qquad
\mathcal H_{\bar i,\bar j}:=\{F\in\mathcal H:i,j\notin F\}.
\]

\begin{lemma}\label{fact} Let
    $\mathcal H\subseteq\binom{[n]}k$ satisfy $\nu(\mathcal H)=s$.  Then
    \[
    s-1\leq\nu(\mathcal H_{\bar i})\leq s
    \]
    for every $i\in[n]$.
\end{lemma}
\begin{proof}
		Since $\mathcal H_{\bar i}\subseteq\mathcal H$, every matching in
		$\mathcal H_{\bar i}$ is also a matching in $\mathcal H$. Hence
		$
		\nu(\mathcal H_{\bar i})\leq s.
		$ Let $\mathcal M$ be an $s$-matching in $\mathcal H$. Since the edges of
		$\mathcal M$ are pairwise disjoint, at most one edge of $\mathcal M$
		contains $i$. Remove this edge if it exists. The remaining edges form a
		matching of size at least $s-1$ in $\mathcal H_{\bar i}$. Therefore,
		$
		\nu(\mathcal H_{\bar i})\geq s-1.
		$
\end{proof}
\begin{lemma}\label{fact2} Let
    $\mathcal H\subseteq\binom{[n]}3$ satisfy $\nu(\mathcal H)=s$.  If $\{i,j\}\in\mathcal K_{3s+1}(\mathcal H)$, then
    \[
    \nu(\mathcal H_{\bar i,\bar j})\leq s-1.
    \]
\end{lemma}
\begin{proof}
Otherwise, let $\mathcal M\subseteq\mathcal H_{\bar i,\bar j}$ be an $s$-matching.  Its union has $3s$ vertices.  Since $d_{\mathcal H}(\{i,j\})\geq3s+1$, some edge $\{i,j,x\}\in\mathcal H$ has $x\notin\bigcup_{F\in\mathcal M}F$.  This edge together with $\mathcal M$ is an $(s+1)$-matching, a contradiction.
\end{proof}
\begin{lemma}\label{maxdegree} Let
    $\mathcal H\subseteq\binom{[n]}3$ satisfy $\nu(\mathcal H)=s$.  If
    \[
    d_{\mathcal K_{3s+1}(\mathcal H)}(\{i\})\geq3s+1
    \]
    for some $i\in[n]$, then $\nu(\mathcal H_{\bar i})=s-1$.
\end{lemma}
\begin{proof}
Lemma~\ref{fact} gives the lower bound $s-1$.  If $\nu(\mathcal H_{\bar i})=s$, let $\mathcal M\subseteq\mathcal H_{\bar i}$ be an $s$-matching.  Its union has $3s$ vertices. Since
\[
d_{\mathcal K_{3s+1}(\mathcal H)}(\{i\})\geq 3s+1,
\]
there exists a vertex
\[
j\notin\bigcup_{F\in\mathcal M}F
\quad\text{such that}\quad
\{i,j\}\in\mathcal K_{3s+1}(\mathcal H).
\]
Equivalently, $d_{\mathcal H}(\{i,j\})\geq 3s+1$. Hence
$\mathcal M\subseteq\mathcal H_{\bar i,\bar j}$, contradicting
Lemma~\ref{fact2}.
\end{proof}
\begin{lemma}\label{ineq2}
Let $p\geq2$, $0<b<a$, and let $x_1,\ldots,x_N\in[0,a]$.  If at most $c$ of these numbers exceed $b$ and $\sum_{i=1}^N x_i\leq m$, where $c\geq0$, then
\[
\sum_{i=1}^N x_i^p\leq ca^p+(m-ca)b^{p-1}.
\]
\end{lemma}
\begin{proof}
For $0\leq x\leq b$, we have $x^p\leq b^{p-1}x$.  For $b<x\leq a$, the function $f(x)=x^p-b^{p-1}x$ is increasing, and hence
\[
x^p\leq b^{p-1}x+a^p-ab^{p-1}.
\]
Summing these pointwise inequalities and using the fact that at most $c$ terms exceed $b$ gives
\[
\sum_{i=1}^N x_i^p
\leq b^{p-1}m+c(a^p-ab^{p-1}),
\]
which is the claimed bound.  The argument also covers $c=0$.
\end{proof}

\begin{lemma}\label{ineq3}
If a graph $G$ satisfies $\nu(G)\leq s$ and $\Delta(G)\leq r$, then
\[
e(G)\leq s(r+1).
\]
\end{lemma}
\begin{proof}
	Vizing's theorem \cite{vizing1964} states that every finite simple graph
	of maximum degree $\Delta$ admits a proper edge-colouring with at most
	$\Delta+1$ colours.  Consequently, $E(G)$ can be partitioned into at most
	\[
	\Delta(G)+1\leq r+1
	\]
	colour classes, each of which is a matching.  Since $\nu(G)\leq s$, every
	such colour class contains at most $s$ edges.  Therefore,
	\[
	e(G)\leq s\bigl(\Delta(G)+1\bigr)\leq s(r+1),
	\]
	as required.
\end{proof}

We now prove the explicit $3$-uniform result.  The case $s=1$ is included in the same induction and requires no separate estimate.

\noindent\textit{Proof of Theorem \ref{th k=3}.}
We use induction on $s$, simultaneously for all positive integers $p$, and interpret $\mathcal H_{n,3,0}$ as the empty family.  The assertion for $s=0$ is immediate.  If $\nu(\mathcal H)\leq s-1$, then the induction hypothesis gives
\[
co_p(\mathcal H)\leq co_p(\mathcal H_{n,3,s-1})
<co_p(\mathcal H_{n,3,s}),
\]
so we may assume $\nu(\mathcal H)=s$.  For $n\geq6s+3$, the canonical candidate is strictly larger than the clique candidate, since
\[
|\mathcal H_{n,3,s}|\geq s\binom{n-s}{2}
\geq s\binom{5s+3}{2}>\binom{3s+2}{3};
\]
indeed, the difference between the last two quantities is
\[
s\binom{5s+3}{2}-\binom{3s+2}{3}=2s(2s+1)^2>0.
\]
Consequently, for $p=1$, the result, including uniqueness, follows from the exact $3$-uniform matching theorem and its equality analysis \cite[Section~13]{frankl2017maximum}.  Hence assume $p\geq2$.

By Lemma~\ref{fact}, $\nu(\mathcal H_{\bar i})\in\{s-1,s\}$ for every $i\in[n]$.

\noindent\textbf{Case 1.}  Suppose that $\nu(\mathcal H_{\bar i})=s-1$ for some $i$; take $i=1$.  For $E\in\binom{[2,n]}2$,
\[
d_{\mathcal H}(E)=d_{\mathcal H_{\bar1}}(E)
+\mathbf1_{E\cup\{1\}\in\mathcal H}
\leq d_{\mathcal H_{\bar1}}(E)+1,
\]
whereas $d_{\mathcal H}(E)\leq n-2$ for pairs $E$ containing $1$.  Since $n-1\geq6(s-1)+3$, the induction hypothesis, applied to each $1\leq j\leq p$, yields
\begin{align*}
co_p(\mathcal H)
&\leq(n-1)(n-2)^p+\binom{n-1}{2}
+\sum_{j=1}^p\binom pj co_j(\mathcal H_{\bar1})\\
&\leq(n-1)(n-2)^p+\binom{n-1}{2}
+\sum_{j=1}^p\binom pj co_j(\mathcal H_{n-1,3,s-1})\\
&=co_p(\mathcal H_{n,3,s}),
\end{align*}
where the last equality is Lemma~\ref{eq1}.  If equality holds, then every pair containing $1$ has codegree $n-2$, so the full star centred at $1$ is contained in $\mathcal H$; moreover, equality in the $j=p$ term and the induction hypothesis give $\mathcal H_{\bar1}\cong\mathcal H_{n-1,3,s-1}$.  Thus $\mathcal H\cong\mathcal H_{n,3,s}$.

\noindent\textbf{Case 2.}  Suppose that $\nu(\mathcal H_{\bar i})=s$ for every $i\in[n]$, and set
\[
\mathcal K:=\mathcal K_{3s+1}(\mathcal H).
\]
Lemma~\ref{lem1} gives $\nu(\mathcal K)\leq s$.  By the contrapositive of Lemma~\ref{maxdegree}, $\Delta(\mathcal K)\leq3s$, and hence Lemma~\ref{ineq3} gives
\[
|\mathcal K|\leq c:=s(3s+1).
\]
Put
\[
a:=n-2,\qquad b:=3s,\qquad
m:=3\left(\binom n3-\binom{n-s}3\right),\qquad
D:=\binom n2-\binom{n-s}2.
\]
The exact $3$-uniform matching theorem gives
\[
\sum_{E\in\binom{[n]}2}d_{\mathcal H}(E)=3|\mathcal H|\leq m.
\]
All pair-codegrees are at most $a$, and at most $c$ of them exceed $b$.  Lemma~\ref{ineq2} therefore yields
\begin{equation}\label{three graph moment bound}
co_p(\mathcal H)\leq ca^p+(m-ca)b^{p-1}.
\end{equation}

We claim that the right-hand side of \eqref{three graph moment bound} is strictly smaller than $Da^p$.  First note that $a>b$ and $D>c$ for $n\geq6s+3$; indeed,
\[
a-b=n-3s-2\geq3s+1,
\qquad
D-c=\frac{s(2n-7s-3)}2\geq\frac{s(5s+3)}2>0.
\]
If $m\leq ca$, then
\[
ca^p+(m-ca)b^{p-1}\leq ca^p<Da^p.
\]
Suppose now that $m>ca$.  It is enough to verify
\[
(D-c)a^2>(m-ca)b,
\]
because multiplication by $a^{p-2}$ and the inequality $a>b$ then give
\[
(D-c)a^p>(m-ca)ba^{p-2}\geq(m-ca)b^{p-1}
\]
for every $p\geq2$.  Put $t=n-6s-3\geq0$.  Direct expansion, retaining the factor $3$ in the definition of $m$, gives
\begin{align*}
2\bigl((D-c)a^2-(m-ca)b\bigr)
=s\bigl(&2t^3+(20s+7)t^2+(51s^2+40s+8)t\\
&+15s^3+24s^2+14s+3\bigr)>0.
\end{align*}
This proves the claim in a directly checkable form.  Combining it with Lemma~\ref{ineq1}, we obtain
\[
co_p(\mathcal H)<D(n-2)^p\leq co_p(\mathcal H_{n,3,s}).
\]
Thus equality is impossible in Case~2, while Case~1 gives precisely the stated equality family. \hfill$\square$

The argument above is uniform in $p$: after the reduction to $p=2$, the threshold depends only on $s$.

%%%%%%%%%%%%%%%%%%%%%%%%	
\vskip.2cm
\section*{Acknowledgements}
M. Cao is supported by the National Natural Science Foundation of China (Grant 12301431) and the Beijing Natural Science Foundation (Grant 1262010).  M. Lu is supported by the National Natural Science Foundation of China (Grant 12571372).
%{\bf Acknowledgments}
%Many thanks to the anonymous referee for his/her many helpful comments and suggestions, which
%have considerably improved the presentation of the paper.

\end{document}